\documentclass{amsart}
\usepackage{amsmath,amsthm,amssymb,IMjournal}
\usepackage{hyperref}
\usepackage{xcolor}

\begin{document}

\numberwithin{equation}{section}

\newtheorem{theorem}{Theorem}[section]
\newtheorem{corollary}{Corollary}[section]
\newtheorem{remark}{\textit{Remark}}[section]
\newtheorem{lemma}{\textit{Lemma}}[section]
\newtheorem{cor}{\textit{Corollary}}[section]
\newtheorem{example}{\textit{Example}}[section]
\newtheorem{conjecture}{\textit{Conjecture}}[section]
\newtheorem{proposition}{\textit{Proposition}}[section]

\def\b{\boldsymbol}
\def\note#1{ {\sffamily \textcolor{blue}{#1}} }
\def\Reg{\text{Reg}}

\title{Lanczos-like algorithm for the time-ordered exponential: The $\ast$-inverse problem.}

\author{\|Pierre-Louis |Giscard|, Universit\'e du Littoral C\^{o}te d’Opale, EA2597-LMPA-Laboratoire de Math\'ematiques Pures et Appliqu\'ees Joseph Liouville, Calais, France. Email: \href{mailto:giscard@univ-littoral.fr}{giscard@univ-littoral.fr},\\
        \|Stefano |Pozza|, Faculty of Mathematics and Physics, Charles University, Sokolovsk\'a 83, 186 75 Praha 8, Czech Republic. Email: \href{mailto:pozza@karlin.mff.cuni.cz}{pozza@karlin.mff.cuni.cz}}
  
  \rec{\today}      

\abstract
The time-ordered exponential of a time-dependent matrix $\mathsf{A}(t)$ is defined as the function of $\mathsf{A}(t)$ that solves the first-order system of coupled linear differential equations with non-constant coefficients encoded in $\mathsf{A}(t)$. The authors recently proposed the first Lanczos-like algorithm capable of evaluating this function. This algorithm relies on inverses of time-dependent functions with respect to a non-commutative convolution-like product, denoted $\ast$. Yet, the existence of such inverses, crucial to avoid algorithmic breakdowns, still needed to be proved. Here we constructively prove that $\ast$-inverses exist for all non-identically null, smooth, separable functions of two variables. As a corollary, we partially solve the Green's function inverse problem which, given a distribution $G$, asks for the differential operator whose fundamental solution is $G$.
Our results are abundantly illustrated by examples.
\endabstract
%
%


\thanks
We thank Francesca Arrigo, Des Higham, Jennifer Pestana, and Francesco Tudisco for their invitation to the  University of Strathclyde, without which this work would not have come to fruition. This work has been supported by Charles University Research program No. UNCE/SCI/023 and the 2019 \textsc{Alcohol} project ANR-19-CE40-0006.
\endthanks

\section{Introduction: Time-ordered Exponential and $\ast$-Lanczos Algorithm}\label{AlgoSection}
\subsection{Context}
Consider the $N\times N$ matrix $\mathsf{A}(t')$ depending on the real-time variable $t'\in I\subseteq \mathbb{R}$.
The time-ordered exponential of $\mathsf{A}(t')$ is defined as the unique solution $\mathsf{U}(t',t)$ of the system of coupled linear differential equations with non-constant coefficients 
\begin{equation}\label{FundamentalSystem}
\mathsf{A}(t') \mathsf{U}(t',t)=\frac{d}{dt'}\mathsf{U}(t',t), \quad \mathsf{U}(t,t)=\mathsf{Id},\text{ for all }t\in I,
\end{equation}
with $t\leq t' \in I$ and $\mathsf{Id}$ the identity matrix. Under the assumption that $\mathsf{A}$ commutes with itself at all times, i.e., $\mathsf{A}(\tau_1)\mathsf{A}(\tau_2)-\mathsf{A}(\tau_2)\mathsf{A}(\tau_1)=\b{0}$ for all $\tau_1,\tau_2 \in I$, then the time-ordered exponential is an ordinary matrix exponential $\mathsf{U}(t',t)=\exp\left(\int_t^{t'} \mathsf{A}(\tau)\, \text{d}\tau\right)$.
In general, however, $\mathsf{U}$ has no known explicit form in terms of $\mathsf{A}$.
In spite of its widespread applications throughout physics, mathematics, and engineering, the time-ordered exponential function is still very challenging to calculate. 
Recently P.-L. G. and S. P. proposed the first Lanczos-like algorithm \cite{GiscardPozza2019} capable of evaluating  $\b{w}^{H}\mathsf{U}(t',t)\b{v}$ for any two vectors $\b{w}, \b{v}$ with $\b{w}^{H}\b{v}=1$, where $\b{w}^{H}$ is the Hermitian transpose of $\b{w}$. The algorithm inherently relies on a non-commutative convolution-like product, denoted by $\ast$, between time-dependent functions and necessitates the calculation of inverses with respect to this product. The purpose of the present contribution is to constructively establish the existence of these inverses.
More generally, these results answer the Green's function inverse problem: namely, given a function $G$ of two variables, what is the differential operator whose fundamental solution is $G$? Here, our results are valid even when the function $G$ is a smooth and separable function of two variables $G(t',t)$ rather than depending solely on $t'-t$; a simpler case for which the $\ast$-product reduces to a convolution and the solution is obtained from standard Fourier analysis.

Before these results can be presented, we recall the definition and properties of the product utilized. 

\subsection{$\ast$-Product}\label{ProdDef}
Let $t$ and $t'$ be time variables in an interval $I\subseteq \mathbb{R}$. Let $f_1(t',t)$ and $f_2(t',t)$ be time-dependent generalized functions. We define the convolution-like $\ast$ product between $f_1(t',t)$ and $f_2(t',t)$ as
\begin{equation}\label{eq:def:*}
  \big(f_2 * f_1\big)(t',t) := \int_{-\infty}^{\infty} f_2(t',\tau) f_1(\tau, t) \, \text{d}\tau.
\end{equation}
From this definition, we find the identity element with respect to the $\ast$-product to be the Dirac delta distribution, $1_\ast:=\delta(t'-t)$. Observe that the $*$-product is not, in general, a convolution but may be so when both $f_1(t',t)$ and $f_2(t',t)$ depend only on the difference $t'-t$.

As a case of special interest for the $\ast$-Lanczos algorithm, consider the situation where $f_1(t',t):=\tilde{f}_1(t',t)\Theta(t'-t)$ and $f_2(t',t):=\tilde{f}_2(t',t)\Theta(t'-t)$, where $\Theta(\cdot)$ stands for the Heaviside theta function (with the convention $\Theta(0)=1$).
Here and in the rest of the paper, the tilde indicates that  $\tilde{f}$ is an ordinary function. Then the $\ast$-product between $f_1,f_2$ simplifies to
\begin{align*}
  \big(f_2 * f_1\big)(t',t) &= \int_{-\infty}^{\infty} \tilde{f}_2(t',\tau) \tilde{f}_1(\tau, t)\Theta(t'-\tau)\Theta(\tau-t) \, \text{d}\tau,\\ &=\Theta(t'-t)\int_t^{t'} \tilde{f}_2(t',\tau) \tilde{f}_1(\tau, t) \, \text{d}\tau,
\end{align*}
 which makes calculations involving such functions easier to carry out. 

The $\ast$-product extends directly to time-dependent matrices by using the ordinary matrix product between the integrands in \eqref{eq:def:*}
 (see \cite{GiscardPozza2019} for more details). It is also well defined for functions that depend on less than two-time variables. Indeed, consider a generalized function $f_3(t')$, then
\begin{align*}
\big(f_3 \ast f_1\big)(t',t)&= f_3(t')\int_{-\infty}^{+\infty}  f_1(\tau, t) \, \text{d}\tau,\\
\big(f_1\ast f_3\big)(t',t)&=\int_{-\infty}^{+\infty}  f_1(t',\tau)f_3(\tau) \, \text{d}\tau.
\end{align*}
where $f_1(t',t)$ is defined as before. Hence the time variable of $f_3(t')$ is treated as the left time variable of a doubly time-dependent generalized function. This observation extends straightforwardly to constant functions.

\subsection{$\ast$-Lanczos algorithm}
As shown in \cite{Giscard2015}, if $\tilde{\mathsf{A}}(t')$ is a time-dependent matrix with bounded entries for every $t' \in I$, then the related time-ordered exponential $\mathsf{U}(t',t)$ can be expressed as
\begin{equation}\label{OrderedExp}
\mathsf{U}(t',t)=\Theta(t'-t)\int_{t}^{t'} \mathsf{R}_{\ast}(\tilde{\mathsf{A}})(\tau,t)\,\text{d}\tau.
\end{equation}
Here $\mathsf{R}_{\ast}$ is the $\ast$-resolvent, defined as
\begin{align*}
\mathsf{R}_{\ast}(\tilde{\mathsf{A}})&:=\big(\mathsf{Id}1_\ast-\tilde{\mathsf{A}}\big)^{\ast-1},\\
&\textcolor{white}{:}=\mathsf{Id}\,1_\ast+\sum_{k> 0}\tilde{\mathsf{A}}^{\ast k}.
\end{align*}

\begin{table}
     \noindent\fbox{
\parbox{0.95\textwidth}{\begin{center}
\parbox{0.9\textwidth}{
\bigskip

 \noindent \underline{Input:} A complex time-dependent matrix $\mathsf{A}$, and
                  complex vectors $\b{v},\b{w}$ such that $\b{w}^H\b{v} = 1$.

 \noindent \underline{Output:} Coefficients $\alpha_0,\cdots,\, \alpha_{n-1}$ and $\beta_0,\cdots,\, \beta_{n-1}$ defining the matrix $\mathsf{T}_n$ of Eq.~(\ref{eq:tridiag}) which satisfies Eq.~(\ref{AjTjResult}).
  \begin{align*}
  & \textrm{Initialize: } \b{v}_{-1}=\b{w}_{-1}=0, \, \b{v}_0 = \b{v} \, 1_*, \, \b{w}_0^H = \b{w}^H 1_*. \\
  &    \alpha_0 = \b{w}^H \mathsf{A} \, \b{v}, \\
  &    \b{w}_1^H = \b{w}^H \mathsf{A} -  \alpha_{0}\, \b{w}^H,\\
  &    \b{\widehat v}_{1} = \mathsf{A}\, \b{v} - \b{v}\,\alpha_{0}, \\
  &    \beta_1 = \b{w}^H \mathsf{A}^{*2} \, \b{v} - \alpha_0^{*2}, \\
  &    \qquad\textrm{If } \beta_1 \textrm{ is not $*$-invertible, then stop, otherwise}, \\
  &     \b{v}_1 = \b{\widehat v}_1 * \beta_1^{*-1}, \\[1em]
  & \textrm{For } n=2,\dots \\ 
  &    \qquad \quad \alpha_{n-1} = \b{w}_{n-1}^H * \mathsf{A} * \b{v}_{n-1}, \\
  &    \qquad \quad \b{w}_{n}^H = \b{w}_{n-1}^H * \mathsf{A} -  \alpha_{n-1} * \b{w}_{n-1}^H - \beta_{n-1}*\b{w}_{n-2}^H, \\
  &    \qquad \quad \b{\widehat v}_{n} = \mathsf{A} * \b{v}_{n-1} - \b{v}_{n-1}*\alpha_{n-1} - \b{v}_{n-2}, \\
  &     \qquad \quad \beta_n = \b{w}_{n}^H * \mathsf{A} * \b{v}_{n-1}, \\
  &     \qquad \qquad\quad \textrm{If } \beta_n \textrm{ is not $*$-invertible, then stop, otherwise}, \\
  &     \qquad \quad  \b{v}_n = \b{\widehat v}_n * \beta_n^{*-1}, \\
  & \textrm{end}.
  \end{align*}
  }\end{center}}}
  \caption{The $*$-Lanczos Algorithm of \cite{GiscardPozza2019}.}\label{algo:*lan}
  \end{table} 
  
Now we can recall the results \cite{GiscardPozza2019} pertaining to the time-ordered-exponential. Let $\mathsf{A}(t',t):=\tilde{\mathsf{A}}(t')\Theta(t'-t)$ with $\tilde{\mathsf{A}}(t')$ a $N\times N$ time-dependent matrix. The $\ast$-Lanczos algorithm of Table \ref{algo:*lan} produces a sequence of tridiagonal matrices $\mathsf{T}_n$, $1\leq n\leq N$, of the form
\begin{equation}\label{eq:tridiag}
    \mathsf{T}_n := \begin{bmatrix} 
            \alpha_0 & 1_* &          & \\ 
            \beta_1  & \alpha_1 & \ddots & \\ 
                     & \ddots   & \ddots & 1_* \\
                     &           & \beta_{n-1} & \alpha_{n-1}
          \end{bmatrix},
\end{equation}
and such that the matching moment property is achieved:
\begin{theorem}[\cite{GiscardPozza2019}]\label{thm:mmp}
   Let $\mathsf{A},\b{w},\b{v}$ and $\mathsf{T}_n$ be as described above, then
   \begin{equation}\label{AjTjResult}
        \b{w}^H (\mathsf{A}^{*j})\, \b{v} = \b{e}_1^H (\mathsf{T}_n^{*j})\,\b{e}_1, \quad \text{ for } \quad j=0,\dots, 2n-1.
   \end{equation}
\end{theorem}
In particular, for $n=N$, we have the exact expression
$$
\b{w}^H\mathsf{U}(t',t)\b{v}=\Theta(t'-t)\int_{t}^{t'} \mathsf{R}_{\ast}(\mathsf{T}_{n})_{1,1}(\tau,t)\,\text{d}\tau,
$$
while for $n<N$, the right-hand side yields an approximation to the time-ordered exponential. 
The method of path-sum \cite{Giscard2015} then gives explicitly
\begin{equation}\label{PSresult}
\mathsf{R}_{\ast}(\mathsf{T}_n)_{1,1}(t',t) 
= \Big(1_\ast - \alpha_0-\big(1_\ast-\alpha_1-(1_\ast-...)^{\ast-1}\ast\beta_2\big)^{\ast-1}\ast\beta_1\Big)^{\ast-1}.
\end{equation}
The $\alpha_j$ and $\beta_j$ appearing in the $\mathsf{T}_n$ matrices are produced by the $\ast$-Lanczos procedure through recurrence relations. A crucial step in the algorithm is the $\ast$-inversion of the $\beta_j$, i.e, the calculation of a distribution $\beta^{\ast-1}_j$ such that $\beta^{\ast-1}_j\ast \beta_j = \beta_j \ast \beta^{\ast-1}_j=1_\ast$. The paper \cite{GiscardPozza2019} assumed the existence of such $\ast$-inverses.  However, if a  $\beta^{\ast-1}_j$ fails to exist, then the algorithm suffers a breakdown. 

Under the assumption that all entries of the input matrix $\mathsf{A}(t')$ are smooth functions, we conjectured in \cite{GiscardPozza2019} that all the coefficients $\alpha_{n-1}$ and $\beta_n$ in $\ast$-Lanczos algorithm are of the form  $\alpha_{n-1}=\tilde{\alpha}_{n-1}\Theta(t'-t)$, $\beta_{n}=\tilde{\beta}_{n}\Theta(t'-t)$, with $\tilde{\alpha}_{n-1}$ and $\tilde{\beta}_{n}$ separable functions (see definition in \S\ref{InversesSection}) and smooth in both time variables.
 This conjecture is justified not only by our experiments but also by observing that the set of the separable functions smooth in both $t'$ and $t$ is closed under $\ast$-product, summation, and differentiation. In spite of these encouraging observations, proving the conjecture is surprisingly difficult as nothing a priori precludes the $\alpha_{n-1}$ and $\beta_n$ coefficients produced by the $\ast$-Lanczos algorithm from being arbitrary distributions. Nonetheless, under the conjecture and its assumptions, we
 prove here \emph{in a constructive way} that the algorithmic breakdowns due to $\beta^{\ast-1}_j$ failing to exist cannot happen unless $\beta_j$ is identically null. More generally, we show that the $\ast$-inverse $f^{\ast-1}$ of a function $f(t',t) = \tilde{f}(t',t) \Theta(t'-t)$ can be obtained when $\tilde{f}$ is smooth, not identically null, and separable. Note that here and later, the existence of a $\ast$-inverse means that it exists almost everywhere in $I\times I$.\\[-.7em] 
 
The rest of this article is organized as follows: in \S\ref{InversesSection}, we begin by recalling necessary definitions and properties of separable functions and distributions. In \S\ref{InvOneVar}, we give the $\ast$-inverses of functions of a single variable. We then proceed in \S\ref{InvPoly} with the $\ast$-inverses of all functions that are polynomials in at least one variable. Encouraged by the method underlying these results, we generalize it to construct the $\ast$-inverse of any piecewise smooth separable function in \S\ref{InvSeparable}. Finally, in \S\ref{GreenProb}, we present the relation between our results and the Green's function inverse problem.

\section{Existence and mathematical expression of $\ast$-inverses}\label{InversesSection}
The calculation of $\ast$-inverses of functions $f(t',t)$ carries the gist of the difficulty inherent in obtaining explicit expressions for time-ordered exponentials.
In general, given an arbitrary ordinary function $\tilde{f}(t',t)$ and barring any further assumption, the $\ast$-inverse of $f(t',t)= \tilde{f}(t',t)\Theta(t'-t)$ cannot be given explicitly.\footnote{Practical numerical questions pertaining to the behavior of $\ast$-inverses under time discretization will be discussed in detail elsewhere. As observed in \cite{GiscardPozza2019}, a time-discretized $\ast$-inverse is always computable using an ordinary matrix inverse.} In this section, we show that the $\ast$-inverse $f^{\ast-1}$ is indeed accessible from the solution of an ordinary linear differential equation provided that $\tilde{f}(t',t)$ is a separable function that is smooth in $t,t'$ and not identically null. 
A function $\tilde{f}(t',t)$ is separable if and only if there exist ordinary functions $\tilde{a}_i$ and $\tilde{b}_i$ with
$$
    \tilde{f}(t',t)=\sum_{i=1}^{k}\tilde{a}_i(t')\tilde{b}_i(t).
$$

We begin by recalling important properties of the Dirac delta distribution and its derivatives $\delta^{(j)}$. The Dirac delta derivatives are characterized by the relation expounded by Schwartz \cite{schwartz1978}, $\int_{-\infty}^{\infty}\delta^{(j)}(q)f(q)\,\text{d}q=(-1)^j f^{(j)}(0)$. From this we get that $\ast$-multiplication by $\delta^{(j)}$ acts as a derivative operator
\begin{align*}
    \big(\delta^{(j)}\ast f\big)(t',t)  &= \int_{-\infty}^{\infty} \delta^{(j)}(t'-\tau)\,f(\tau,t)\,\text{d}\tau,\\
    &=-\int_{\infty}^{-\infty}\delta^{(j)}(q)\,f(t'-q,t)\,\text{d}q, \\
    &=(-1)^j  \frac{\partial^j}{\partial q^j}f(t'-q,t)\big|_{q=0},\\
    &=f^{(j,0)}(t',t),
\end{align*}
while we have $f\ast \delta^{(j)} = (-1)^j f^{(0,j)}$. The notation $f^{(j,k)}(\tau,\rho)$ stands for the $j$th $t'$-derivative and $k$th $t$-derivative  of  $f(t',t)$ evaluated at $t'=\tau, t=\rho$. From now on, we omit the $t'-t$ arguments of the Heaviside $\Theta$ functions and Dirac deltas when necessary to alleviate the equations.\\[-.7em]

For functions of the form $f(t',t)=\tilde{f}(t',t)\Theta(t'-t)$, the derivatives resulting from the $\ast$-action of $\delta^{(j)}$ are taken in the sense of distributions:\begin{subequations}\label{eq:deltader} \begin{align}
       \delta^{(j)} * f(t',t) &= \tilde{f}^{(j,0)}(t',t) \Theta + \tilde{f}^{(j-1,0)}(t,t) \delta +
       \dots +
       \tilde{f}(t,t) \delta^{(j-1)},\\
      f(t',t)\ast  \delta^{(j)}  &=(-1)^j\left( \tilde{f}^{(0,j)}(t',t) \Theta + \tilde{f}^{(0,j-1)}(t',t') \delta +
       \dots +
       \tilde{f}(t',t') \delta^{(j-1)}\right) ;
   \end{align}
\end{subequations}
see \cite[Chapter 2, \S~2]{schwartz1978}.
Finally, we note the following identities between distributions for $j\geq 0$
\begin{subequations}\label{RelDeltas}
\begin{align}
\tilde{f}(t')\delta^{(j)}(t'-t)&=(-1)^j \big(\tilde{f}(t)\delta(t'-t)\big)^{(0,j)},\\
\tilde{f}(t)\delta^{(j)}(t'-t)&= \big(\tilde{f}(t')\delta(t'-t)\big)^{(j,0)},
\end{align}
\end{subequations}
where $\tilde{f}$ is an ordinary function.

\subsection{Functions of a single time variable}\label{InvOneVar}
The $\ast$-inverse of functions of a single time variable times a Heaviside function are easy to find explicitly: 
\begin{proposition}\label{OneTimeInverses}
Let $a(t',t):=\tilde{a}(t')\Theta(t'-t)$ and $b(t',t):=\tilde{b}(t)\Theta(t'-t)$ so that $\tilde{a}$ and $\tilde{b}$ are differentiable, and not identically null over $I$. Then
\begin{equation*}
    a^{\ast-1}(t',t) =\frac{\partial}{\partial t'}\frac{\delta(t'-t)}{\tilde{a}(t')}, \quad b^{\ast-1}(t',t) = - \frac{\delta'(t'-t)}{\tilde{b}(t)}.
\end{equation*}
\end{proposition}
\begin{proof}
Since $\tilde{a}(t')$ is an ordinary function and $a(t',t)=\tilde{a}(t')\Theta(t'-t)$, Eqs.~(\ref{eq:deltader}) and \cite[Chapter 2, \S~2]{schwartz1978} give
\begin{equation*}
    \big(a\ast \delta' \big)(t',t)  =  \tilde{a}(t')\delta(t'-t),
\end{equation*}
as $\Theta^{(0,1)}(t'-t)=-\delta(t'-t)$.
We deduce that  
$
    a^{\ast -1}(t',t) = \delta'(t'-t) \ast \left(\tilde{a}(t')\delta(t'-t)\right)^{\ast -1}.
$
The $\ast$-inverse of $\tilde{a}(t')\delta(t'-t)$ is the solution $x(t',t)$ of the equation $\tilde{a}(t')\delta(t'-t)\ast x(t',t)=\delta(t'-t)$,
i.e., $x(t',t) = \delta(t'-t)/\tilde{a}(t')$, from which we get the expression  
$$
a^{\ast -1}(t',t)=\delta'(t'-t)\ast \frac{\delta(t'-t)}{\tilde{a}(t')}.
$$
An analogous proof yields the inverse $b^{\ast-1}$.
\end{proof}

 Proposition~\ref{OneTimeInverses} is particularly useful to determine the $\ast$-inverse of products of functions of a single time variable such as those of \cite{GiscardPozza2019}. We give two detailed examples of this below:\\[-.5em]
\begin{example}
Let us determine the $\ast$-inverse of $(t'-t)\Theta$. To this end, we remark that $(t'-t)\Theta = \Theta \ast \Theta$ and thus
$$
\left((t'-t)\Theta\right)^{\ast-1}=\big(\Theta^{\ast-1}\ast \Theta^{\ast-1}\big).
$$
Since $\Theta=1\times \Theta$, the $\ast$-inverse of $\Theta$ is immediately provided by Proposition~\ref{OneTimeInverses} as $\Theta^{\ast-1}=\delta'$.
Then $\Theta^{\ast-1}\ast\Theta^{\ast-1}=\delta''$, whose $\ast$-action on a test function $f(t',t)$ is
\begin{align*}
\big(\Theta^{\ast-1}\ast\Theta^{\ast-1}\ast f\big)(t',t)
&=f^{(2,0)}(t',t).
\end{align*}
\end{example}

\begin{example}
Let us find the left and right actions of the $\ast$-inverse of $\beta(t',t):=2\big(\sin(t')-\sin(t)\big)+(t'-t)$ on test functions. We note that $\beta= b_2\ast b_1$ with $b_2(t')=\Theta(t'-t)$ and $b_1(t')=2(\cos(t')+1)\Theta(t'-t)$. Hence by Proposition~\ref{OneTimeInverses}, the left action of the inverse on a test function $f(t',t)$ is
\begin{align*}
\beta^{\ast-1}\ast f&=b_1^{\ast-1}\ast b_2^{\ast-1}\ast f=\frac{\partial}{\partial t'}\left[\frac{1}{2(\cos(t')+1)}\frac{\partial}{\partial t'}f(t',t)\right],\\
&=\frac{\sin(t')}{2(\cos(t')+1)^2}\frac{\partial}{\partial t'}f(t',t)+\frac{1}{2(\cos(t')+1)}\frac{\partial^2}{\partial t'^2}f(t',t),
\end{align*}
and its right action is
\begin{align*}
f\ast \beta^{\ast-1}&=f\ast b_1^{\ast-1}\ast b_2^{\ast-1}=\frac{\partial}{\partial t}\left[\frac{1}{2(\cos(t)+1)}\frac{\partial}{\partial t}f(t',t)\right],\\
&=\frac{\sin(t)}{2(\cos(t)+1)^2}\frac{\partial}{\partial t}f(t',t)+\frac{1}{2(\cos(t)+1)}\frac{\partial^2}{\partial t^2}f(t',t).
\end{align*}
\end{example}

\subsection{$\ast$-inverses of polynomials}\label{InvPoly}
The method employed in the proof of Proposition~\ref{OneTimeInverses} relying on differential equations generalizes straightforwardly to polynomials in at least one time variable, here taken to be $t'$. An analogous result can be given for functions that are polynomials in $t$.
\begin{proposition}\label{prop:inv:poly}
 Let $p(t',t)=\tilde{p}(t',t)\Theta(t'-t)$ be so that $\tilde{p}(t',t)$ is a polynomial of degree $k\geq 1$ in $t'$ and is smooth in $t$. If $p(t,t)$ is not identically null over $I$, then
 \begin{equation*}
    p(t',t)^{*-1} = x(t',t) * \delta^{(k+1)}(t'-t),    
 \end{equation*}
 where $x(t',t)=\tilde{x}(t',t)\Theta(t'-t)$ is the solution of the linear homogeneous ordinary differential equation in $t$
 \begin{align*}
  \sum_{j=0}^k(-1)^j\tilde{p}^{(k-j,0)}(t,t)\tilde{x}^{(0,j)}(t',t) = 0 ,
 \end{align*}
 with the boundary conditions 
 \begin{equation*}
     \tilde{x}^{(0,k-1)}(t',t') = \frac{(-1)^k}{\tilde{p}(t',t')}, ~ \tilde{x}^{(0,k-2)}(t',t') = 0, ~ \dots, ~ \tilde{x}(t',t') = 0.
 \end{equation*}
 \end{proposition}
\begin{proof}
Observe that $p(t',t)$ is a piecewise smooth function, and, as a function of $t'$, it has a discontinuity located at $t'=t$. Since furthermore, $\tilde{p}(t',t)$ is of degree $k$ in $t'$, Eq.~\eqref{eq:deltader} gives
 \begin{align*}
     \big(\delta^{(k+1)}\ast p\big)(t',t) 
     &=\sum_{j=0}^{k}\tilde{p}^{(k-j,0)}(t,t)\delta^{(j)}(t'-t).
 \end{align*}
 Hence $
    p(t',t)^{*-1} = x(t',t) * \delta^{(k+1)}(t'-t),    
$
 where $x(t',t)$ is the generalized function satisfying
 \begin{equation}\label{eq:ode}
     x(t',t) \ast \left (\sum_{j=0}^k \tilde{p}^{(k-j,0)}(t,t)\delta^{(j)}(t'-t) \right) = \delta(t'-t).
 \end{equation}
 Now let us assume that the solution $x(t',t)$ takes the form $x(t',t)=\tilde{x}(t',t)\Theta(t'-t)$ with  $\tilde{x}(t',t)$ a smooth function of $t$. Then we get, for $j=0,\dots, k$, 
 \begin{align*}
     x(t',t) \ast \tilde{p}^{(k-j,0)}(t,t)\delta^{(j)} &=\\
     &\hspace{-15mm}\tilde{p}^{(k-j,0)}(t,t)(-1)^j\left( \tilde{x}^{(0,j)}(t',t)\Theta+\sum_{\ell=0}^{j-1}\tilde{x}^{(0,j-1-\ell)}(t',t')\delta^{(\ell)}\right).
 \end{align*}
Thus Eq.~\eqref{eq:ode} can be rewritten as the system:
\begin{align*}
\left\{
\begin{array}{l}
 \displaystyle \sum_{j=0}^k(-1)^j\tilde{p}^{(k-j,0)}(t,t)\tilde{x}^{(0,j)}(t',t) = 0, \\
 \displaystyle \sum_{j=1}^{k}(-1)^j \tilde{p}^{(k-j,0)}(t,t)\tilde{x}^{(0,j-1)}(t',t') = 1, \\
 \displaystyle \sum_{j=2}^{k}(-1)^j \tilde{p}^{(k-j,0)}(t,t)\tilde{x}^{(0,j-2)}(t',t') =  0, \\
 \displaystyle \quad\vdots \qquad\, \qquad \quad \quad \quad \,\vdots \\
 \displaystyle (-1)^k\tilde{p}(t,t) \tilde{x}(t',t')   = 0. 
   \end{array} \right .
\end{align*}
As $\tilde{p}(t,t)$ is not identically null, the last $k-1$ equations imply
$\tilde{x}^{(0,j)}(t',t') = 0$ for $j=0,\dots,k-2$.
Moreover, since by Eq.~(\ref{RelDeltas}) we have $\tilde{p}(t,t)\delta(t'-t) = \tilde{p}(t',t')\delta(t'-t)$, the second equation becomes $(-1)^k\tilde{p}(t',t')\tilde{x}^{(0,k-1)}(t',t') = 1$. Since the set of zeros of $\tilde{p}(t',t')$ is made of isolated points, the ordinary differential equation above has a solution almost everywhere (more precisely, $\tilde{x}(t',t)$ is defined for $t',t \in I \setminus \{\tau : \tilde{p}(\tau,\tau) = 0\}$). Thus assuming $x(t',t)$ to be of the form $\tilde{x}(t',t)\Theta(t'-t)$ with $\tilde{x}$ smooth in $t$ is a consistent choice, which concludes the proof.
\end{proof}
~\\[-1.5em]

\begin{remark}\label{rem:ptt}
If $\tilde{p}(t,t)$ is identically null over $I$, then
\begin{equation*}
    \delta'(t'-t)\ast p(t',t) = \tilde{p}^{(1,0)}(t',t)\Theta(t'-t),
\end{equation*}
since $p$ is continuous at $t'=t$. Hence we can apply Proposition \ref{prop:inv:poly} to $\tilde{p}^{(1,0)}(t',t)\Theta(t'-t)$. In the further case in which all $\tilde{p}^{(j)}(t,t)=0$ are identically null for $j=0,\dots,k-1$ and $\tilde{p}^{(k)}(t,t)$ is a constant $\alpha$, the $\ast$-inverse is obtained noting that
\begin{equation*}
    \delta^{(k+1)}(t'-t)\ast p(t',t) = \alpha \, \delta(t'-t).
\end{equation*}
These considerations show that the condition $\tilde{p}(t,t)\neq 0$ is not necessary for $p^{\ast-1}(t',t)$ to exist. Rather the  condition is that $p(t',t)$ itself must not be identically zero.\\[-.5em]
\end{remark}

\begin{example}Let us determine the $\ast$-inverse of the polynomial $p(t',t):=\big(t'-2t\big)\Theta(t'-t)$. 
Following Proposition~\ref{prop:inv:poly} we have
$
p(t',t)^{\ast-1}=x(t',t)\ast\delta''
$,
where $x(t',t)=\tilde{x}(t',t)\Theta$ and $\tilde{x}(t',t)$ solves
$$
\tilde{x}(t',t)+t \tilde{x}^{(0,1)}(t',t)=0,~~\tilde{x}(t',t')=\frac{1}{t'}.
$$
This gives $\tilde{x}(t',t)=1/t$ and thus
$$
p(t',t)^{\ast-1}=\left(\frac{1}{t}\Theta\right)\ast \delta''
=\frac{2}{t^3}\Theta-\frac{1}{t'^2}\delta+\frac{1}{t'}\delta'.
$$
We can now verify that this works as expected
\begin{align*}
p(t',t)^{\ast-1}\ast p(t',t) &=\left(\frac{1}{t}\Theta\right)\ast\delta''\ast (t'-2t)\Theta=\left(\frac{1}{t}\Theta\right)\ast\big((t'-2t)\Theta\big)^{(2,0)},\\
&=\left(\frac{1}{t}\Theta\right)\ast\left(\delta-t\delta'\right)=\frac{1}{t}\Theta-(-1)t\left(\frac{-1}{t^2}\Theta+\frac{1}{t'}\delta\right),\\
&=\frac{t}{t'}\delta = \delta,
\end{align*}
where the last equality follows by virtue of Eq.~(\ref{RelDeltas}). Now
\begin{align*}
p(t',t)\ast p(t',t)^{\ast-1} &=(t'-2t)\Theta\ast\left(\frac{1}{t}\Theta\right)\ast\delta'',\\
&=-(t'-t)\Theta\ast \delta''=-(-1)^2\big((t'-t)\Theta\big)^{(0,2)},\\
&=-\big(0-\delta-0\delta'\big)=\delta.
\end{align*}
\end{example}

A technique similar to the one used in the proof of Proposition \ref{prop:inv:poly} can be applied to a more general class of functions. For instance, whenever differentiating leads to an expression like
\begin{equation*}
    \delta^{(k)} \ast f(t',t) = \tilde{h}(t)f(t',t) + g(t',t),
\end{equation*}
the expression can be rewritten as
\begin{equation*}
    \big(\delta^{(k)}- \tilde{h}(t)\delta\big) \ast f(t',t) = g(t',t).
\end{equation*}
Then we can go on with a further combination of differentiations until there is no Heaviside function left on the right-hand side of the above equality. 
In particular, such a technique can be used when dealing with commonly encountered exponential or trigonometric functions. 

\subsection{$\ast$-inverses of piecewise smooth separable functions}\label{InvSeparable}
 The strategy used in the proof of Proposition~\ref{OneTimeInverses} can be extended to give $\ast$-inverses in the much more general case of functions which are separable and piecewise smooth in both time variables over the interval $I$.\\[-.7em] 

\begin{theorem}\label{InvHolo}
   Consider a function $f(t',t) := \tilde{f}(t',t) \Theta(t'-t)$ with $\tilde{f}(t',t)$ 
   a smooth function in $I \times I$, and so that $\tilde{f}(t,t)$ is not identically null. Assume that there exists a distribution $L(t',t):=\sum_{j=0}^{k+1} \tilde{g}_j(t')\delta^{(j)}$ with $\tilde{g}_j(t')$ smooth functions depending only on $t'$ and $\tilde{g}_{k+1}\neq0$, such that
   \begin{equation}\label{eq:Lf=0}
       L(t',t) \ast \tilde{f}(t',t) =0.
   \end{equation}
   Then, if $k>0$, the $\ast$-inverse of $f$ is
   \begin{align*}
f^{\ast-1} =\tilde{r}_{-1}(t',t)\Theta+\sum_{m=0}^k \tilde{r}_m(t') \delta^{(m)},
\end{align*}
with the smooth functions
\begin{align*}
\tilde{r}_{-1}(t',t)&:=\sum_{j=0}^{k+1}(-1)^j\tilde{y}_j^{(0,j)}(t',t),\\
\tilde{r}_{m\geq0}(t')&:=\sum_{j=m+1}^{k+1}(-1)^j \tilde{y}_j^{(0,j-1-m)}(t',t'),
\end{align*}
where $\tilde{y}_j(t',t):=\tilde{x}(t',t)\tilde{g}_j(t)$ and $\tilde{x}(t',t)$ is the solution of the linear homogeneous ordinary differential equation in $t$ 
\begin{equation*}
\sum_{m=0}^k\tilde{h}_m(t)\tilde{x}^{(0,m)}(t',t)=0,
\end{equation*}
with boundary conditions
\begin{equation*}
\tilde{x}^{(0,k-1)}(t',t') = \left(\tilde{h}_k(t')\right)^{-1},\, 
    \tilde{x}^{(0,k-2)}(t',t') = 0, \dots, \tilde{x}(t',t') = 0. 
\end{equation*}
In these expressions, $\tilde{h}_m(t)$ are smooth functions given by
$$
   \tilde{h}_m(t):=\sum_{j=m+1}^{k+1}\sum_{\ell=m}^{j-1}\binom{\ell}{m}(-1)^\ell \tilde{f}^{(j-\ell-1,0)}(t,t)\tilde{g}_j^{(\ell-m)}(t).
   $$
 If instead $k=0$, the $\ast$-inverse of $f$ is trivially given by
 $$
 f^{\ast-1}(t',t)=\frac{1}{\tilde{g}_1(t')\tilde{f}(t',t')} L(t',t).
 $$
\end{theorem}

Inverting the role of $t'$ and $t$, a completely similar theorem is proven by changing all left $\ast$-multiplications by $\delta^{(j)}$ with right multiplications and vice-versa. In this situation, $\tilde{x}$ satisfies a linear homogeneous ordinary differential equation in $t'$, and the boundary conditions involve the variable $t$.

\begin{proof}
By $\ast$-multiplying $L$ by $f$, we get
   \begin{equation}\label{eq:Lf}
       L(t',t) \ast f(t',t) = \sum_{j=0}^{k+1} \tilde{g}_j(t')\left(\delta^{(j)} \ast f(t',t)\right),
   \end{equation}
   where 
   \begin{equation*}
       \delta^{(j)} * f(t',t) = \tilde{f}^{(j,0)}(t',t) \Theta + \tilde{f}^{(j-1,0)}(t,t) \delta +
       \dots +
       \tilde{f}(t,t) \delta^{(j-1)}.
   \end{equation*}
      Therefore \eqref{eq:Lf}  evaluates to
   \begin{equation*}
       L \ast f = \big(L \ast \tilde{f}\,\big)\Theta+\sum_{j=0}^{k}\left(\sum_{\ell=j+1}^{k+1}\tilde{g}_\ell(t')\tilde{f}^{(\ell-j-1,0)}(t,t)\right)\delta^{(j)}.
   \end{equation*}
   Noting that $L \ast \tilde{f}=0$ and by applying the transformation of Eq.~(\ref{RelDeltas}) to $\tilde{g}_\ell(t')\delta^{(j)}$, the last equation can be further expressed in the form $L\ast f=\sum_{m=0}^{k}(-1)^m \tilde{h}_m(t)\delta^{(m)}$ with the smooth functions
   $$
   \tilde{h}_m(t):=\sum_{j=m+1}^{k+1}\sum_{\ell=m}^{j-1}\binom{\ell}{m}(-1)^\ell \tilde{f}^{(j-\ell-1,0)}(t,t)\tilde{g}_j^{(\ell-m)}(t).
   $$
   Assume that $x(t',t) = \tilde{x}(t',t) \Theta(t'-t)$, with $\tilde{x}$ smooth function of $t$,  is the $\ast$-inverse of $L\ast f$. Then it should satisfy
   \begin{equation}\label{eq:xinf:eq}
        x(t',t) \ast \left( \sum_{m=0}^{k}(-1)^m  \tilde{h}_m(t)\delta^{(m)} \right) = \delta .
   \end{equation}
   We now proceed similarly to the proof of Proposition~\ref{prop:inv:poly}. Since again
   \begin{align*}
       x(t',t) \ast \delta^{(m)}  
       &=(-1)^m\tilde{x}^{(0,m)}(t',t)\Theta+(-1)^m\sum_{j=0}^{m-1}\tilde{x}^{(0,m-1-j)}(t',t')\delta^{(j)},
   \end{align*}
   it follows that 
   \begin{align*}
\big(  x \ast L\ast f\big)(t',t) &= \sum_{m=0}^k\tilde{h}_m(t)\left(\tilde{x}^{(0,m)}(t',t)\Theta+\sum_{j=0}^{m-1}\tilde{x}^{(0,m-1-j)}(t',t')\delta^{(j)}\right),\\
&=\sum_{m=0}^k\tilde{h}_m(t)\tilde{x}^{(0,m)}(t',t)\Theta+\sum_{m={\color{blue}1}}^k\sum_{j=0}^{m-1}\tilde{h}_m(t)\tilde{x}^{(0,m-1-j)}(t',t')\delta^{(j)},\\
&=\sum_{m=0}^k\tilde{h}_m(t)\tilde{x}^{(0,m)}(t',t)\Theta+\sum_{j=0}^{k-1}\sum_{m=j+1}^k\tilde{h}_m(t)\tilde{x}^{(0,m-1-j)}(t',t')\delta^{(j)}.
   \end{align*}
Now by Eq.~(\ref{RelDeltas}), $\tilde{h}_m(t)\delta^{(j)}=\big(\tilde{h}_m(t')\delta\big)^{(j,0)}=\sum_{n=0}^j \binom{j}{n}\tilde{h}_m^{(j-n)}(t')\delta^{(n)}$, we have   
\begin{align*}  
&\sum_{j=0}^{k-1}\sum_{m=j+1}^k\tilde{h}_m(t)\tilde{x}^{(0,m-1-j)}(t',t')\delta^{(j)}\\
&\hspace{15mm}=\sum_{j=0}^{k-1}\sum_{m=j+1}^k\tilde{x}^{(0,m-1-j)}(t',t')\sum_{n=0}^j \binom{j}{n}\tilde{h}_m^{(j-n)}(t')\delta^{(n)},\\
&\hspace{15mm}=\sum_{n=0}^{k-1}\left(\sum_{j=n}^{k-1}\sum_{m=j+1}^k\tilde{x}^{(0,m-1-j)}(t',t')\binom{j}{n}\tilde{h}_m^{(j-n)}(t')\right)\delta^{(n)}.
\end{align*}   
Hence Eq.~\eqref{eq:xinf:eq} becomes equivalent to the ordinary homogenous linear differential equation in $t$
\begin{equation}\label{ODExHolo}
\sum_{m=0}^k\tilde{h}_m(t)\tilde{x}^{(0,m)}(t',t)=0,
\end{equation}
with boundary conditions
\begin{align*}
&\sum_{j=0}^{k-1}\sum_{m=j+1}^k\tilde{x}^{(0,m-1-j)}(t',t')\tilde{h}_m^{(j)}(t')=1,\\
&\sum_{j=n}^{k-1}\sum_{m=j+1}^k\tilde{x}^{(0,m-1-j)}(t',t')\binom{j}{n}\tilde{h}_m^{(j-n)}(t')=0,\quad n=1,\dots,k-1.
\end{align*}
For every $t'$ such that $\tilde{h}_k(t')\neq 0$, the last $k-1$ equations imply
$\tilde{x}^{(0,j)}(t',t') = 0$, for $j=0,\dots,k-2$, and $\tilde{x}^{(0,k-1)}(t',t')\tilde{h}_k(t') = 1$.
Thus $\tilde{x}$ is well defined for almost every $t'\in I$ as the unique solution of Eq.~(\ref{ODExHolo}) with the boundary conditions above and the choice of $\tilde{x}$ as a smooth function of $t$ is consistent ($\tilde{x}(t',t)$ is defined for every $t',t \in I \setminus \{\tau : \tilde{f}(\tau,\tau)\tilde{g}_{k+1}(\tau) = 0\}$). 

We can now evaluate $f^{\ast-1}=x\ast L$ explicitly, 
\begin{align*}
f^{\ast-1} &= \sum_{j=0}^{k+1}(-1)^j\big(\tilde{x}(t',t)\tilde{g}_j(t)\big)^{(0,j)}\Theta+\sum_{j=1}^{k+1}(-1)^j\sum_{m=0}^{j-1}\big(\tilde{x}(t',t')\tilde{g}_j(t')\big)^{(0,j-1-m)}\delta^{(m)},\\
&=\tilde{r}_{-1}(t',t)\Theta+\sum_{m=0}^k \tilde{r}_m(t') \delta^{(m)},
\end{align*}
with the smooth functions
\begin{align*}
\tilde{r}_{-1}(t',t)&:=\sum_{j=0}^{k+1}(-1)^j\big(\tilde{x}(t',t)g_j(t)\big)^{(0,j)},\\
\tilde{r}_{m\geq0}(t')&:=\sum_{j=m+1}^{k+1}(-1)^j\big(\tilde{x}(t',t')\tilde{g}_j(t')\big)^{(0,j-1-m)}.
\end{align*}
\end{proof}

\begin{remark}
As explained in Remark \ref{rem:ptt}, the assumption $f(t,t)=0$ is not necessary. We can reformulate the theorem statement so that the condition is $f$ not identically zero on $I$.\\[-.5em]
\end{remark}

The most stringent condition imposed by Theorem~\ref{InvHolo} is the existence of the differential operator $L$ with coefficients that depend only on $t'$. This condition can be made more transparent upon relating it to the class of separable functions.

Let $\tilde{y}_1(t'),\dots, \tilde{y}_{k+1}(t')$ be smooth functions of $t'$, and $\tilde{a}_1(t),\dots,\tilde{a}_{k+1}(t)$ be functions of $t$.
If $\tilde{y}_1(t'),\dots, \tilde{y}_{k+1}(t')$ are linearly independent, equivalently, the related \emph{Wronskian} $W(\tilde{y}_1,\dots,\tilde{y}_{k+1})$ is not identically null, i.e.,
\begin{equation*}
    W(\tilde{y}_1,\dots,\tilde{y}_{k+1}) := 
    \begin{vmatrix} 
    \tilde{y}_1       & \tilde{y}_2       & \dots & \tilde{y}_{k+1} \\
    \tilde{y}_1'      & \tilde{y}_2'      & \dots & \tilde{y}_{k+1}' \\
    \vdots    & \vdots    &       & \vdots \\
    \tilde{y}_1^{(k)} & \tilde{y}_2^{(k)} & \dots & \tilde{y}_{k+1}^{(k)}
    \end{vmatrix} \neq 0,
\end{equation*}
then there exist $L$ as in Theorem \ref{InvHolo} so that $L*\tilde{f} = 0$ for every separable function 
\begin{equation}\label{eq:sepfun}
    \tilde{f}(t',t) = \tilde{a}_1(t)\tilde{y}_1(t') + \dots + \tilde{a}_{k+1}(t)\tilde{y}_{k+1}(t').
\end{equation}
Indeed, the conditions $L*\tilde{y}_j = 0$, for $j=1,\dots,k+1$, give the system
\begin{equation*}
    \begin{bmatrix} 
    \tilde{y}_1      & \tilde{y}_1'       & \dots & \tilde{y}_1^{(k)} \\
    \tilde{y}_2      & \tilde{y}_2'       & \dots & \tilde{y}_2^{(k)} \\
    \vdots   & \vdots     &       & \vdots \\
    \tilde{y}_{k+1}  & \tilde{y}_{k+1}'   & \dots & \tilde{y}_{k+1}^{(k)}
    \end{bmatrix}
    \begin{bmatrix}
    \tilde{g}_0 \\
    \tilde{g}_1 \\
    \vdots \\
    \tilde{g}_k
    \end{bmatrix} =
    -\tilde{g}_{k+1}\begin{bmatrix}
    \tilde{y}_{1}^{(k+1)} \\
    \tilde{y}_{2}^{(k+1)} \\
    \vdots \\
    \tilde{y}_{k+1}^{(k+1)}
    \end{bmatrix},
\end{equation*}
whose solutions exist since the Wronskian is not identically null. In particular, at least one of the solutions has smooth coefficients.
Theorem \ref{InvHolo} thus yields the following corollary for separable functions:

\begin{corollary}
   Let $f(t',t) := \tilde{f}(t',t) \Theta(t'-t)$ with $\tilde{f}(t',t)$ 
   a smooth separable function in $I \times I$ so that $\tilde{f}(t,t)$ is not identically null. Then $f^{\ast-1}$ exists and is given as in Theorem \ref{InvHolo}.
\end{corollary}

\begin{example}
Let us determine the $\ast$-inverse of $f(t',t)=(t'^2-t/t')\Theta(t'-t)$. Since $\tilde{f}(t',t)$ is separable, smooth in both variables, and $\tilde{f}(t,t)$ is not identically null, Theorem~\ref{InvHolo} applies immediately. Setting $L(t',t):=\tilde{g}_0(t') \delta+\tilde{g}_1(t')\delta'+\tilde{g}_2(t')\delta''$ with $\tilde{g}_0(t'):=1$, $\tilde{g}_1(t'):=0$ and $\tilde{g}_2(t'):=-t'^2/2$,
we have $k=1$ and
$
L(t',t)\ast \tilde{f}(t',t)=0.
$
This leads to 
$$
\tilde{h}_0(t):=3t/2,\quad \tilde{h}_1(t):=(t^4+t^2)/2,
$$
which are the only non-identically null functions $\tilde{h}_m$.
The function $\tilde{x}$ is thus the solution of 
$$
3\tilde{x}(t',t)+t(1+t^2)\tilde{x}^{(0,1)}(t',t)=0,\quad
\tilde{x}(t',t')\frac{t'^2}{2}(t'^2+1)=1.
$$
We find
$$
\tilde{x}(t',t)=\frac{2t'(t^2+1)^{3/2}}{(t'^2+1)^{5/2}t^3}.
$$
We verify that  
$$
x\ast L \ast f=\tilde{h}_1(t)\tilde{x}(t',t')\delta=\frac{t^2 \left(t^2+1\right)}{t'^2 \left(t'^2+1\right)}\delta=\delta,
$$
indicating that indeed 
$$
f^{\ast -1}=x\ast L=\left(\frac{2t'(t^2+1)^{3/2}}{(t'^2+1)^{5/2}t^3}\Theta\right) \ast\left(\delta-\frac{t'^2}{2}\delta''\right).
$$
\end{example}

\begin{example}
Let us determine the $\ast$-inverse of $f(t',t)=\cos(t') t\,\Theta(t'-t)$. Since $\tilde{f}(t',t)$ is separable, smooth in both variables, and $\tilde{f}(t,t)$ is not identically null, Theorem~\ref{InvHolo} applies. Furthermore, let $L(t',t):=\delta+\delta''=1\times\delta+0\times \delta'+1\times \delta''$ and observe that
$
L(t',t)\ast \tilde{f}(t',t)=0.
$
It follows that here $k=1$ and $\tilde{x}(t',t)$ is the solution of 
$$
-t\sin(t)\tilde{x}(t',t)-t\cos(t)\tilde{x}^{(0,1)}(t',t)=0,~\tilde{x}(t',t')=\frac{-1}{t'\cos(t')}.
$$
This gives 
$$
\tilde{x}(t',t)=-\frac{\cos(t)}{\cos(t')^2t'},
$$
and, from there,
$$
f^{\ast-1}(t',t)=\frac{\sin(t')}{\cos(t')^2t'}\,\delta-\frac{1}{\cos(t')t'}\,\delta' .
$$
We verify this result
\begin{align*}
f^{\ast-1}(t',t)\ast f(t',t) &= \frac{\sin(t')}{\cos(t')^2t'}\cos(t')t\,\Theta-(-1)\frac{1}{\cos(t')t'}\big(-\sin(t')t\,\Theta+\cos(t) t\, \delta\big),\\
&=\frac{\cos(t)t}{\cos(t')t'}\delta=\delta,
\end{align*}
by virtue of Eq.~(\ref{RelDeltas}). The proof for $ f(t',t)\ast f^{\ast-1}(t',t)=\delta$ is similar.\\[-.3em] 
\end{example}

\begin{example}\label{ex:exp}Let us determine the $\ast$-inverse of $e^{3t'+t}$. We can apply Theorem~\ref{InvHolo}, this time with $L(t',t)=\delta-(1/3)\delta'$, i.e., $k=0$. Thus
$$
\tilde{f}^{\ast-1}(t',t)=-3e^{-4t'}\left(\delta-\frac{1}{3}\delta'\right).
$$
We verify this result immediately
$$
f\ast f^{\ast-1}=-3e^{3t'-3t}\Theta-(-1)\big(-3e^{3t'-3t}\Theta+e^{-t'+t'}\delta\big)=\delta,
$$
and similarly for $f^{\ast-1}\ast f$. 
\end{example}

\begin{remark}
Our results concern $\ast$-inverses of separable functions of the form $f:=\tilde{f} \Theta$, with $\tilde{f}$ separable and smooth in $t'$ and $t$. In particular, they do not extend easily to $\ast$-resolvents, which are $\ast$-inverses of generalized functions of the form $\delta-f$, typically with $f$ as above. Rather, $\ast$-resolvents are best determined as solutions of a linear Volterra integral equation of the second kind with kernel $f$ and inhomogeneity $1_\ast=\delta$,
$$
R_\ast(f):=\big(1_\ast-f\big)^{\ast-1}\Rightarrow R_\ast(f)=\delta +f \ast R_\ast(f).
$$
There is a vast literature on the existence and smoothness of the solutions of such equations \cite{Gripenberg1990, sezer1994, razdolsky2017}, as well as numerous techniques to determine them both analytically and numerically \cite{Tricomi1985, Linz1985, Polyanin2008, GiscardVolterra}. In the context of the $\ast$-Lanczos algorithm, $\ast$-resolvents play a central role in the final step when computing $\mathsf{R}_\ast(\mathsf{T})_{11}$ via Eq.~(\ref{PSresult}), but they can also be profitably exploited when calculating the $\beta^{\ast-1}_j$.
Indeed, for $j \geq 2$, $\beta_j$ can be expressed in the alternative way
\begin{equation}\label{eq:beta:expr}
\beta_j = (\b{w}_j^H + \b{w}_{j-2}^H)* \mathsf{A} * \b{v}_{j-1} - 1_*.
\end{equation}
The advantage of this representation is that it shows $\beta_j^{*-1}$ to be the $\ast$-resolvent of 
$\gamma_j = (\b{w}_j^H + \b{w}_{j-2}^H)* \mathsf{A} * \b{v}_{j-1}$, which gives direct access to all research on solutions of Volterra equations, including for situations where $\beta_j$ does not satisfy the assumptions of Theorem \ref{InvHolo}.
\end{remark}

\section{Relation to the Green's function inverse problem}\label{GreenProb}
Let $G$ be a distribution. The Green's function inverse problem consists in determining an operator $D_G$ whose fundamental solution is $G$, i.e., $D_G\, (G)=\delta$. 
This problem, also known as kernel inverse problem, appears sporadically in the literature when a kernel function $G$ is motivated by external constraints, and the corresponding differential operator is determined from it secondarily; see e.g., in interpolation problems \cite{Bouhamidi2005, Fasshauer2013}.

In the most commonly encountered framework, however, the product utilized is a convolution. Then $D_G$ is found from its Fourier (or Laplace) transform, which is the inverse of the Fourier transform of $G$. The problem considered here is thus more general, the $\ast$-product reducing to a convolution only when the functions involved depend only on the difference between the two-time variables. Here we rather only suppose that $G$ is a non-identically null distribution of the form $G:=\tilde{G}(t',t)\Theta(t'-t)$ such that $\tilde{G}$ is separable and smooth in both time variables. Then the proof of Theorem~\ref{InvHolo} constructively shows that there exists a distribution $G^{\ast-1}$  such that 
$$
G^{\ast-1}\ast G = \delta.
$$
In other terms, the $\ast$-action of $G^{\ast-1}$ on $G$ is identical with the ordinary action of the differential operator $D_G$ whose Green's function is $G$.
In order to give $D_G$ explicitly, observe that
$$
G^{\ast-1}=\tilde{r}_{-1}(t',t)\Theta+\sum_{m=0}^k \tilde{r}_m(t') \delta^{(m)},
$$
with the smooth functions $\tilde{r}_{-1\leq j\leq k}$ defined in Theorem~\ref{InvHolo}. Since for any distribution $f$, $\delta^{(j)}\ast f = f^{(j,0)}$, we get the action of $D_G$ on any distribution $f$ as
$$
D_G\,f=\int_{-\infty}^{\infty}\tilde{r}_{-1}(t',\tau)\Theta(t'-\tau) f(\tau,t)\,\text{d}\tau+\sum_{m=0}^k \tilde{r}_m(t') \frac{\partial^m}{\partial t'^m}f(t',t).
$$
Here recall that should $f$ depend on a single variable or less, then it should be treated as the left time-variable as indicated in Subsection \ref{ProdDef}, that is here $t'$.\\

\section{Conclusion}
The $\ast$-Lanczos algorithm for evaluating time-ordered exponentials relies on the existence of the $\ast$-inverses of the coefficients $\beta_n(t',t)$ produced by the algorithm. Should an inverse fail to exist, the Lanczos procedure suffers a breakdown, and the ordered exponential cannot be evaluated.  Now, under the conjecture of \cite{GiscardPozza2019} and its assumptions, $\beta_n(t',t)=\tilde{\beta}_n(t',t)\Theta(t'-t)$, where $\tilde{\beta}_n(t',t)$ is a separable function, smooth in both $t'$ and $t$. Assuming this to be true, we showed that if $\beta_n(t',t)$ is not identically null, then its $\ast$-inverse exists and the algorithm does not breakdown. Furthermore, we described explicit procedures to obtain the required $\ast$-inverses and illustrated our results with several examples. These procedures relate  $\ast$-inverses to the solutions of linear differential equations with smooth coefficients. 
As a corollary of this work, we solved a generalization of the Green's function inverse problem for piecewise smooth distributions.


\end{document}